\documentclass{amsart}

\usepackage{amssymb,amscd,amsthm,latexsym, amsmath}
\usepackage[all]{xy}

\newtheorem{defi}{Definition}[section]
\newtheorem{theo}{Theorem}[section]
\newtheorem{lemma}{Lemma}[section]
\newtheorem{prop}{Proposition}[section]
\newtheorem{coro}{Corollary}[section]

\def\into{ \rightarrowtail }

\def\splito{ \rightleftarrows }
\def\trio{ \triangleright}

\newdir{ >}{{}*!/-8pt/@{>}}
\def\AA{ \mathbb{A} }
\def\EE{ \mathbb{E} }
\def\CC{ \mathbb{C} }
\def\DD{ \mathbb{D} }
\def\VV{ \mathbb{V} }

\def\HH{ \mathbb{H} }
\def\Pt{ \mathrm{Pt} }
\def\trio{ \triangleright}
\newcommand{\Ker}{ \ensuremath{\mathrm{Ker}} }

\newcommand{\Rng}{\mathsf{Rng}}
\newcommand{\APq}{\mathsf{APq}}
\newcommand{\Alex}{\mathsf{Alex}}
\newcommand{\LAlex}{\mathsf{LAlex}}
\newcommand{\Qnd}{\mathsf{Qnd}}
\newcommand{\PQd}{\mathsf{PQd}}
\newcommand{\LPq}{\mathsf{LPq}}
\newcommand{\Di}{\mathsf{Di}}
\newcommand{\Al}{\mathsf{Al}}
\newcommand{\SkB}{\mathsf{SkB}}
\newcommand{\Slom}{\mathsf{Slom}}
\newcommand{\SSlom}{\mathsf{\Sigma lom}}

\newcommand{\Gp}{\mathsf{Gp}}
\newcommand{\Ab}{\mathsf{Ab}}

\newcommand{\tw}{\mathsf{tw}}
\newcommand{\Set}{\mathsf{Set}}
\newcommand{\UMag}{\mathsf{UMag}}
\newcommand{\Mon}{\mathsf{Mon}}
\newcommand{\ALPq}{\mathsf{ALPq}}

\newcommand{\CoM}{\mathsf{CoM}}
\newcommand{\Sbt}{\mathsf{Sbt}}
\newcommand{\HSt}{\mathsf{HSt}}

\begin{document}
	
	\author{Dominique Bourn}
	
	\title{Hypersubtraction and semi-direct product}
	
	\address{Univ. Littoral C\^ote d'Opale, UR 2597, LMPA, Laboratoire de Math\'ematiques Pures et Appliqu\'ees Joseph Liouville, F-62100 Calais, France}
	\email{bourn@univ-littoral.fr}
	
	\noindent
	
	\begin{abstract}
		In this article, we introduce an extrinsic approach to the notion of semi-direct product, an intrinsic one having been already done in \cite{BJ} and \cite{BJK}. This will led us to focus our attention on two algebraic structures which will allow us to characterize this extrinsic explicitation.
	\end{abstract}
	
	\maketitle
	
	\section*{Introduction}
	
	The notion of {\em semi-direct} product is  well known for groups and associated with any group action of a group $(Y,*)$ on a group $(K,*)$. What is remarkable is that any split epimorphism $(f,s): X\splito Y$ in the category $\Gp$ of groups is obtained, up to isomorphism, by such a semi-direct product. The first attemp to understand this heavy structural fact intrinsically (namely inside the category $\Gp$ itself) was made in \cite{BJ}. Then, it was more elaborated in \cite{BJK} and  completed in \cite{BJK2} with the notion of {\em split extension classifier}, see other precisions in \cite{BBB}. This also led, indirectly, to the remarkable observations on the category $\Gp$ produced in \cite{Gr} with the notion of {\em algebraic exponentiation}. Here we try to understand this notion extrinsincally, namely through the inspection of the properties of the forgetful functor $V: \Gp \to \Set_*$ from groups to pointed sets. We shall investigate the notion of {\em semi-direct index} which insures that any split epimorphim $(f,s)$ is sent, up to isomorphism, on the canonical split epimorphism $(p_{V(Y)},\iota_{V(Y)}): V(Y)\times V(\Ker f) \splito V(Y)$. Then we shall show that, for any left exact conservative functor $U: \CC \to \DD$ between pointed categories and without any condition on $\DD$, the existence of a semi-direct index implies that $\CC$ is necessarily a protomodular category. For all that, we shall be led to focus our attention on two algebraic structures: hyper-S\l om\'inski setting and hypersubtraction. 
	
	This article is organized along the following lines: Section 1 investigates these two structures; it contains some considerations about internal quandles and introduces the notion of prequandle. Section 2 is devoted to brief recalls about protomodular categories and partial Mal'tsev categories, and section 3 to the characterization of the existence of the semi-direct indexes and hyperindexes.

	\section{ILO setting setting and hypersubtraction}
	
	In the section \ref{sd} concerning the semi-direct product, we shall be interested in the isomorphisms $\rho$ making commute the following diagram of split epimorphisms in a {\em pointed} category $\DD$:
	$$
	\xymatrix@=8pt{
		X\times X \ar[dd]_{p_1^X}  \ar[rr]^<<<<<<{\rho}  && {\; X\times X} \ar[dd]_{p_0^X} \\
		&& \\
		X \ar@<-1ex>[uu]_{s_0^X}  \ar[rr]_{=}  && X \ar@<-1ex>[uu]_{\iota_0^X} 
	}
	$$
	where $p_0^X,p_1^X,s_0^X,\iota_0^X$ are respectively the two projections, the diagonal and $(Id_X,0_X)$.
	
	\subsection{ILO settings}
	
	Let us begin, in any category $\DD$, with the isomorphisms $\rho$ as  on the left hand  side:
	$$
	\xymatrix@=8pt{
		X\times X \ar[dd]_{p_1^X}  \ar[rr]^<<<<<<{\rho}  && {\; X\times X} \ar[dd]_{p_0^X} && X\times X \ar[rr]^{d} \ar[dd]_{p_1^X} && X \ar[dd] \\
		&&&& \\
		X   \ar[rr]_{=}  && X && X \ar[rr]  &&  1 
	}
	$$
	or, equivalently, with the data of a binary operation $d: X\times X \to X$
	making  the right hand side diagram a pullback. The inverse $\rho^{-1}$ produces another binary operation $\circ: X\times X \to X$ such that:  i) $d(x\circ z,x)=z$ and ii) $x\circ d(z,x)=z$. These equations mean that $x\circ -$ is the inverse of $d(-,x)$.  We then get: i)+ii) $\iff$ iii):  $x\circ y=t \iff y=d(t,x)$.
	\begin{defi}
		An internal ILO setting in a category $\DD$ is a pair $(X,d)$ of an object $X$ and a binary operation $d$ such that $(p_1^X,d): X\times X \to X\times X$ is an isomorphism. It is equivalent to a triple $(X,d,\circ)$ where $(d,\circ)$ is a pair of binary operations on $X$ satisfying axioms {\rm i)} and  {\rm ii)}.
	\end{defi} 
	ILO is an acronyme for {\em invertible left hand sided operation}. Given any ILO setting $(X,d)$, we shall call $\circ$ the {\em adjoint} binary operation associated with $d$. The map $\rho$ being an isomorphism, a morphism $h: (X,d)\to (Y,d)$ of ILO settings is, indifferently, a $d$-homomorphism or a $\circ$-homomorphism. \\
	\noindent\textbf{Examples}\\
	1) The ILO setting produced by the twisting isomorphism $\tw(x,y)=(y,x)$ is such that $d(x,y)=x$ is the first projection, while $x\circ y=y$ is the second one.\\
	2) Any group $(G,\circ)$ produces an ILO setting with $d(x,y)=y^{-1}\circ x$ and the binary operation $\circ$.\\
	3) The dual of an ILO setting defined by $(X,d,\circ)^{op}=(X,\circ^{op},d^{op})$ is an ILO setting.
	\begin{defi}
		An internal ILO setting $(X,d)$ is called involutive when it is equal to its dual, namely when $x\circ y=d(y,x)$. It is called symmetric when $d(x,y)=d(y,x)$ or $\circ=d$. It is called latin when $(X,d^{op})$ is an ILO setting as well.
	\end{defi}
	The terminology {\em latin} comes from the fact that the table of the binary operation $d$ of a latin ILO setting determines an latin square.  Any symmetric ILO setting is a latin one. The ILO setting associated with a group $(X,\circ)$ is necessarily latin; it is symmetric if and only if the group is abelian; it is involutive if and only if $x^{2}=1$. 
	There are two main classes of ILO setting:\\
	1) the class A such that $d(x,x)=1$ (which is equivalent to $x\circ 1=x$) to which belongs the ILO structure of a group;\\
	2) the class B such that $d(x,x)=x$ (which is equivalent to $x\circ x=x$) to which belongs the example 1.\\
	Clearly an ILO setting $(X,d,1)$ belonging to $A\cap B$ is reduced to the singleton $1$.
	\begin{prop}\label{assos}
		Let $(X,d,\circ)$ be an $ILO$ setting. The following conditions are equivalent:
		\begin{align*}
			1) \;\; {\rm the \; law}\;  \circ \; {\rm is \; associative}\; ; \;\; \;\;\;\;\; & \;\;\;\;\;\;  2)\;\; d(y,z)\circ d(x,y)=d(x,z)\; ;\\
			3) \;\; d(d(x,z),d(y,z))=d(x,y);\;\;\; &  \;\;\;\;\;\; 4)\;\;  d(x,y)\circ t=d(x\circ t,y).
		\end{align*}
	\end{prop} 
	\proof
	Suppose 1). Then $z\circ (d(y,z)\circ d(x,y))=(z\circ d(y,z))\circ d(x,y)=y \circ d(x,y)=x=z\circ d(x,z)$. Whence 2).\\
	Suppose 2). Then $d(y,z)\circ (d(x,z)\circ d(y,z))=d(x,z)$, while $d(y,z)\circ d(x,y)=d(x,z)$ whence 3).\\
	We get 4) if and only if $d(d(x\circ t,y),d(x,y))
	=t$. From 3) we have:\\ $d(d(x\circ t,y),d(x,y))=d(x\circ t,x)=t$.\\
	We have 1) if and  only if $y\circ z=d((x\circ y)\circ z,x)$. From 4) we get: $d((x\circ y)\circ z,x)=d(x\circ y,x)\circ z=y\circ z$. 
	\endproof
	\begin{prop}\label{commut}
		Let $(X,d,\circ)$ be an $ILO$ setting. The following conditions are equivalent:
		{\rm 1)} the law $\circ$ is commutative; {\rm 2)} $x=d(x\circ y,y)$; {\rm 3)} $d(y,d(y,x))=x$.
	\end{prop} 
	\proof
	By iii), we have 1), namely $y\circ x=x\circ y \iff$ 2).
	We have $d(y,d(y,x))=x \iff d(y,x)\circ x=y=x\circ d(y,x)$. It is equivalent to commutativity of $\circ$, setting $t=d(y,x)$.
	\endproof
	\subsection{Class A: hyper-S\l om\'inski settings}
	
	Now let us come back to our starting point and to the pointed context:
	$$
	\xymatrix@=8pt{
		X\times X \ar[dd]_{p_1^X}  \ar[rr]^<<<<<<{\rho}  && {\; X\times X} \ar[dd]_{p_0^X} && X\times X \ar[rr]^{d} \ar[dd]_{p_1^X} && X \ar[dd] \\
		&& \\
		X \ar@<-1ex>[uu]_{s_0^X}  \ar[rr]_{=}  && X \ar@<-1ex>[uu]_{\iota_0^X}  && X \ar@<-1ex>[uu]_{s_0^X} \ar[rr]  &&  1 \ar@<-1ex>[uu]
	}
	$$
	namely with an isomorphism $\rho$ of split epimorphisms or with a pullback of split epimorphims as on the right hand side. This adds to i) and ii) the axiom iv) $d(x,x)=1$.
	As term identities of a variety, axioms iv) and ii) were considered in \cite{Slo}. According to \cite{BJ}, they are the characterization, in the case $n=1$, of a pointed protomodular variety. Whence the following
	\begin{defi}\label{sslom}
		An internal S\l om\'inski setting in a category $\EE$ is a quadruple $(X,d,\circ,1)$ of an object $X$ and two  binary operations satisfying {\rm iv)} $d(x,x)=1$ and  {\rm ii)} $x\circ d(z,x)=z$. A hyper-S\l om\'inski setting is the same data satisfying axioms  {\rm iv), ii)} and {\rm i)} $d(x\circ z,z)=x$ or, equivalently, it is an ILO setting $(X,d,1)$ satisfying $d(x,x)=1$.
	\end{defi} 
	Accordingly, the ILO settings belonging to the class A coincide with the hyper-S\l om\'inski settings. The ILO setting associated with a group is a hyper-S\l om\'inski setting.
	\begin{prop}
		Given any S\l omi\'nsky setting $(X,d,\circ,1)$, we get: 
		1) $x\circ 1=x$ (the element $1$ is a right unit); 2) $1\circ d(x,1)=x$; 3) $x\circ d(1,x)=1$ ($x$ has an inverse on the right hand side); and  4) $x=y \iff d(x,y)=1$.
	\end{prop}
	\proof
	Only the last point needs a checking: $y=y\circ 1=y\circ d(x,y)=x$.
	\endproof
	We get examples of internal S\l om\'inski and hyper-S\l om\'inski settings with:
	\begin{prop}\label{Slomab}
		Let $\AA$ be a finitely complete additive category. Then any internal S\l om\'inski setting $(X,d)$ in $\AA$ is determined by a triple $(X,f,g)$ where $(f,g)$ is a pair of $X$-endomorphims such that $g.f=-Id_\AA$. We have then: $d(x,y)=f(x-y)$ and $x\circ y=x+g(y)$. A S\l om\'inski setting in $\AA$ is an hyper-S\l om\'inski setting if and only if $f$ is invertible and $g=f^{-1}$.
	\end{prop}
	\proof
	For any map $d: X\times X \to X$ in $\AA$ satisfying $d(x,x)=0$, we get:\\
	$d(x,y)=d((x-y,0)+(y,y))=d(x-y,0)+ d(y,y)=d(x-y,0)$. Setting $d(x,0)=f(x)$, we get $d(x,y)=f(x-y)$. So, $d(0,y)=-f(y)$.
	
	Setting $0\circ x=g(x)$, from $x\circ y=(x+0)\circ(0+y)=(x\circ 0)+(0\circ y)$, we get $x\circ y=x+g(y)$. Now, from $x\circ d(y,x)=x$, we get $x+gf(y-x)=y$; whence $g.f=Id_X$.
	
	Since any additive category is protomodular, $(X,d)$ is a hyper-S\l om\'inski setting if and only if $d(-,0)=f$ is an isomorphism, see Proposition \ref{iloprot}. 
	\endproof 
	When $\EE$ is finitely complete, we shall denote by $\Slom\EE$ and  $\SSlom\EE$ the categories of internal S\l om\'inski and hyper-S\l om\'inski settings. They are both protomodular  categories.
	
	\subsection{Hypersubtractions}
	We shall be interested as well by the following more constraining commutations on the left hand side:
	$$
	\xymatrix@=8pt{
		X \ar@<1ex>[rr]^{s_0^X} \ar[dd]_{=} && X\times X \ar[rr]_{d} \ar[ll]^{p_1^X} \ar[dd]_{\rho}  && X \ar[dd]_{=}\ar@<-1ex>[ll]_{\iota_0^X} && X\times  X \ar[dd]_{p_1^X} \ar[rr]_d  && X \ar[dd] \ar@<-1ex>[ll]_{\iota_0^X}  \\
		&& \\
		X \ar@<1ex>[rr]^{\iota_0^X} && X\times X \ar[rr]_{p_1^X} \ar[ll]^{p_0^X}   &&  X \ar@<-1ex>[ll]_{\iota_1^X} &&  X \ar@<-1ex>@{ >->}[uu]_{s_0^X}  \ar[rr]_{} && 1\ar@<-1ex>[ll]_{} \ar@<-1ex>[uu]_{}
	}
	$$
	which is equivalent to the fact that the right hand side square is a pullback of split epimorphisms. The axiom added to a hyper-S\l om\'inski setting is $\rho.\iota_0^X=\iota_1^X$, namely  v) $d(x,1)=x$.
	Recall that a subtraction in the sense of \cite{UR} is binary operation $d$ satisfying iv) $d(x,x)=1$ and v) $d(x,1)=x$. Whence the following
	\begin{defi}
		A hypersubtraction on an object $X$ in a category $\EE$ is a hyper-S\l om\'inski setting $(X,d,1)$ satisfying $d(x,1)=x$, or equivalently $1\circ x=x$.
	\end{defi}
	When $\EE$ is finitely complete, we shall denote by $\Sbt\EE$ and  $\HSt\EE$ the categories of objects endowed with internal subtractions and hypersubtractions. $\Sbt\EE$ is a subtractive category in the sense of \cite{ZJ}, while obviously $\HSt\EE\subset \SSlom\EE$ is a protomodular one. From Proposition \ref{Slomab}, we get:
	\begin{coro}
		Let $\AA$ be an additive category. Then $\HSt\AA\simeq\AA$.
	\end{coro}
	\proof
	According to the notation of this proposition, from $d(x,0)=x$, the isomorphism $f$ is $Id$.
	\endproof 
	Whence the string of inclusions: $\Gp\EE \subset \HSt\EE \subset \SSlom\EE\subset \Slom\EE$ of protomodular categories. As such, they determine a subcategory of abelian  objects, see Proposition \ref{abprot}.
	The respective subcategories $\Ab(\Slom\EE)\simeq\Slom(\Ab\EE)$, $\Ab(\SSlom\EE)\simeq \SSlom(\Ab\EE)$ are described by Proposition \ref{Slomab}. By the previous corollary we get $\Ab\HSt\EE=\HSt\Ab=\Ab$.
	
	\subsection{Class B: prequandles}
	
	It is worth saying a word about the class B which make commutative the following diagram:
	$$
	\xymatrix@=8pt{
		X\times X \ar[dd]_{p_1^X}  \ar[rr]^<<<<<<{\rho}  && {\; X\times X} \ar[dd]_{p_0^X} \\
		&& \\
		X \ar@<-1ex>[uu]_{s_0^X}  \ar[rr]_{=}  && X \ar@<-1ex>[uu]_{s_0^X} 
	}
	$$
	First, recall the following definition, introduced independly \cite{Jo} and \cite{Mat}, in relationship with the three Reidemeister moves in knot theory, see also \cite{Bq}:
	\begin{defi}
		A quandle is a set $X$ endowed with binary operation  $\trio: X\times X \to X$ which is idempotent and such that $-\trio x$ is $\trio$-automorphism, namely such that $(x\trio y)\trio z=(x\trio z)\trio(y\trio z)$. 
	\end{defi}
	\noindent\textbf{Examples}\\
	1) The example 1 of ILO setting is a quandle, namely a {\em trivial} quandle.\\
	2) Any group $(G,\cdot)$ produces a quandle with the law $x\trio_Gy=y\cdot x\cdot y^{-1}$.\\
	3) With any pair $((A,+,0),f)$ of an abelian group and a $+$-automorphism $f$ is associated the {\em Alexander} quandle on $A$ defined by: $x\trio_f y=f(x)+y-f(y)$. 
	
	When $\EE$ is a left exact category, define $\Alex\EE$ as the category whose objects are the pair $(X,f)$ of an object $X$ in $\EE$ and an isomorphism $f:X\to X$ in $\EE$ and whose morphisms are those maps $h: X\to Y$ which commute with these isomorphisms. When $\AA$ is additive, the category $\Alex\AA$ is additive as well. The Alexander quandles  defines a left exact conservative functor $\Al: \Alex\Ab \to \Qnd$, which is  consequently faithful.
	
	It is then meaningful to  introduce the following definition for the ILO settings of class B:
	\begin{defi}
		An ILO setting of class B (namely, an idempotent ILO setting) will be called a prequandle.
	\end{defi}
	We shall denote by $\circ_{\trio}$, the adjoint operation of $\trio$. 
	The category of internal prequandles in $\EE$ will be denoted by $\PQd\EE$ and the category of internal quandles by $\Qnd\EE$. Since the law $\trio$ of a prequandle is idempotent, any element $x$ of a prequandle $(X,\trio)$ determines a map $x: 1 \to (X,\trio)$ in $\PQd$, and given any map $h: (X,\trio) \to (Y,\trio)$, any inverse image $h^{-1}(y)$ is a prequandle. 
	
	We have $\Ab(\Qnd)\simeq \Alex\Ab$, see \cite{Bq}. We get the same result for prequandles:
	\begin{prop}\label{aba}
		The functor $\Al: \Alex\Ab \to \Ab(\PQd)$ is an isomorphism.
	\end{prop}
	\proof
	The functor $\Al: \Alex\Ab \to \PQd$ actually takes its values in $\Ab(\PQd)$ since any $\trio_f: (X,+)\times (X,+) \to (X,+)$ is a group homomorphism:\\
	$(x+y)\trio_f(x'+y')=f(x+y)+(x'+y')-f(x'+y')$, while:
	$(x\trio x')+(y\trio_fy')=f(x)+x'-f(x')+f(y)+y'-f(y')$. They are equal since  $f$ is a group homomorphism.
	
	Now given any internal abelian group: $(X,\trio)\times (X,\trio)\to (X,\trio)$ in the category $\PQd$, we get $(x+x')\trio(y+y')=(x\trio y)+(x'\trio y')$. With $x'=0=y'$, we check that $f(x)=x\trio 0$ defines a bijective group homomorphism. With $x=a-b$, $x'=b=y'$ and $y=0$, we get $a\trio b=((a-b)\trio 0)+(b\trio b)= f(a-b)+b=a\trio_fb$. So we get $\Al((X,+),f)=((X,\trio),+)$, with $f=-\trio 0$. 
	\endproof
	
	More generally, let $\AA$ be any left exact additive category and $(X,f)$ any object of $\Alex\AA$. Define the internal binary operation: $\trio_f= X\times X \stackrel{f\times (Id_X-f)}{\longrightarrow} X\times  X \stackrel{+}{\rightarrow} X$ which gives an internal  prequandle structure on $X$ in $\AA$ whose adjoint operation $\circ$ is given by $x\circ y=y\trio_{f^{-1}} x$. So, $\circ=\trio_{f^{-1}}^{op}$. Whence a left exact conservative functor $\Al_\AA: \Alex\AA\to\PQd\AA$.
	\begin{prop}\label{abprq}
		Given any finitely complete additive category $\AA$, the functor $\Al_\AA$ is an isomorphism which is factorized through the inclusion $\Qnd\AA\subset \PQd\AA$ and produces the isomorphisms 
		$\PQd\AA=\Qnd\AA\simeq\Alex\AA$.
	\end{prop}
	\proof
	In the additive setting, the definition of $\trio_f$ necessarily satifies the axiom of a quandle, whence the factorization $\Al_\AA: \Alex_\AA\to \Qnd\AA\subset \PQd\AA$.\\
	Now, let $\trio$ be any internal binary idempotent operation on $X$ in $\AA$. Let us set $f(x)=x\trio 0=f(x)$. We get:
	$(x+x')\trio(y+y')=(x\trio y)+(x'\trio y')$. From $(x,y)=(x-y,0)+(y,y)$, we get $x\trio y=(x-y\trio 0)+(y\trio y)=f(x-y)+y=f(x)+y-f(y)$. If, moreover, $\trio$ is underlying a prequandle, the map $f=-\trio 0$ is an isomorphism, and we get $\trio=\trio_f$.
	\endproof
	In other words, the only internal prequandles in an additive category are the Alexander quandles. 
	
	\smallskip  
	
	\noindent\textbf{Remark.} Strangely enough, the additive context produces four isomorphic reduction: $\SSlom\AA\simeq\Alex\AA\simeq \PQd\AA\simeq \Qnd\AA$.
	
	\section{Some recalls about protomodular and Mal'tsev categories}
	
	On the one hand, we mentioned that any category $\Slom\EE$ or $\SSlom\EE$ of internal S\l om\'in\'ski or internal hyper-S\l om\'in\'ski settings is protomodular. It is a strong structural property which requires some recalls.
	
	On the other hand, the structural aspect of the category $\PQd\EE$ of internal prequandles is weaker than protomodularity, but far from being insignificant. It is related to the notion of Mal'tsev category which will require some recalls as well.  
	
	\subsection{Protomodular categories}
	\begin{defi}\cite{Bprot}
		A protomodular category $\EE$ is a finitely category such that, given any pair of commutative squares of vertical split epimorphisms: 
		$$
		\xymatrix@=20pt{
			\bullet \ar[d] \ar[r] & \bullet \ar[d] \ar[r]  & \bullet \ar[d] \\
			\bullet \ar@<-1ex>@{ >->}[u]  \ar[r] & \bullet \ar[r] \ar@<-1ex>@{ >->}[u] & \bullet \ar@<-1ex>@{ >->}[u] 
		}
		$$
		the right hand side square is a pullback as soon as so are the left hand side one and the whole rectangle.
	\end{defi}
	The major examples are the categories $\Gp$ of groups, $\Rng$ of rings, and any category of $R$-algebras when $R$ is a ring. Any additive category is protomodular. The protomodular varieties are characterized in \cite{BJ}. Any protomodular category is a Mal'tsev one. The main point is that in any {\em pointed protomodular category} we get all the classical homological lemmas, see \cite{BB}. We shall produce further examples in the next section.
	
	Given any finitely complete category,  denote by $\Pt\EE$ the category whose objects are the split epimorphisms $(f,s):X\splito Y$ in $\EE$ and whose morphisms are the commutative squares between them. Let $\P_\EE: \Pt\EE \to \EE$ be the functor associating with any split epimorphism $(f,s)$ its codomain $Y$; it is a fibration (called {\em fibration of points}) whose cartesian maps are the pullbacks; we denote by $\Pt_Y\EE$ the pointed fiber of split epimorphisms above $Y$.
	It is easy to check that $\EE$ is protomodular if and only if any base-change functor of the fibration $\P_\EE$ is conservative. When $\EE$ is pointed, it is enough that any base-change functor $\alpha_Y^*:\Pt_Y\EE \to \EE$ along the initial  maps (namely any "kernel functor") is conservative.
	We systematically used in the previous section the following
	\begin{prop}\cite{BB}\label{abprot}
		Let $\EE$ be a pointed protomodular category. On any object $X$ there is atmost one structure of internal unitary magma which is necessarily an internal abelian group. When it is the case, we say that $X$ is an abelian object. In this way, we get a full inclusion $\Ab\EE\subset \EE$, where $\Ab\EE$ is an additive category.
	\end{prop}
	\begin{prop}\label{iloprot}
		Let $\EE$ be a pointed protomodular category. Then $(X,d)$ is an internal hyper S\l om\'inski setting if and only if $d(-,0)$ is an isomorphism.
	\end{prop}
	\proof
	Apply  the definition axiom to the following diagram:
	$$
	\xymatrix@=20pt{
		X \ar[d] \ar[r]^{\iota_0^X} & X\times X \ar[d]_{p_1^X} \ar[r]^d  & X \ar[d] \\
		\bullet \ar@<-1ex>@{ >->}[u]  \ar[r] & X \ar[r] \ar@<-1ex>@{ >->}[u]_{s_0^X} & \bullet \ar@<-1ex>@{ >->}[u] 
	}
	$$
	where the left hand side diagram is a kernel diagram. So, the right hand side square is a pullback, if and only if $d.\iota_0^X$ is an isomorphism.
	\endproof
	\begin{prop}
		Let $\EE$ be a protomodular category. We get $\Slom\EE=\Slom\Ab\EE$, and $\SSlom\EE=\SSlom\Ab\EE\simeq \Alex(\Ab\EE)$.
	\end{prop}
	\proof
	Let $(X,d,\circ,1)$ be a S\l om\'inski setting in $\EE$. Then $p(x,y,z)=x\circ d(z,y)$ produces an internal ternary operation in $\EE$ satisfying $p(x,x,z)=x\circ d(z,x)=z$ and $p(x,y,y)=x\circ d(y,y)=x\circ 1=x$. Then the associated internal unitary magma on $X$ defined by $x*y=p(x,1,y)$ makes $X$ an abelian object in $\EE$, and thus produces a S\l om\'inski setting in the subcategory $\Ab\EE$. Whence $\Slom\EE=\Slom\Ab\EE$. And consequenly $\SSlom\EE=\SSlom\Ab\EE\simeq \Alex(\Ab\EE)$ by Proposition \ref{Slomab}.
	\endproof
	
	\subsection{Structural aspect of prequandles}
	The structural aspect of the category $\PQd\EE$ of internal prequandles is partially related to the  notion of Mal'tsev category.
	\begin{defi}\cite{CLP}
		A Mal'tsev category is a finitely complete category in which any internal reflexive relation is an equivalence relation.
	\end{defi}
	Any protomodular and, a fortiori, any additive category is a  Mal'tsev one.
	\begin{prop}
		Let $\EE$ be a Mal'tsev  category. Then any internal prequandle $(X,\trio)$ in $\EE$ is a quandle. Whence $\PQd\EE=\Qnd\EE$.
	\end{prop}
	\proof
	We have to check $(x\trio y)\trio z=(x\trio z)\trio (y\trio z)$. It is enough to check it when $x=y$ and $y=z$. We indeed get: $(x\trio x)\trio z=x\trio z=(x\trio z)\trio (x\trio z)$; and:
	$(x\trio z)\trio z=(x\trio z)\trio (z\trio z)$.
	\endproof
	There is a characterization of Mal'tsev categories through the fibration of points:
	\begin{prop}\cite{BB}
		A finitely complete category $\EE$ is a Mal'tsev one if and only if, given any pair $((f,s),(g,t))$ of split epimorphisms with common codomain $Z$, the canonical pair $(\iota_X^t,\iota_Y^s)$ of inclusions toward the pullback $X\times_ZY$
		$$
		\xymatrix@=6pt{
			X\times_Z  Y \ar[dd]_{p_X} \ar[rrr]_{p_Y}  &&& Y \ar[dd]_g \ar@<-1ex>@{ >->}[lll]_{\iota_Y^s} \\
			&&{(M)\;\;\;\;\;\;\;}&& \\
			X \ar@<-1ex>@{ >->}[uu]_{\iota_X^t}  \ar[rrr]_{f} &&& Z\ar@<-1ex>[lll]_{s} \ar@<-1ex>[uu]_{t}
		}
		$$
		is jointly strongly epic.
	\end{prop}
	\noindent where:
	\begin{defi}
		A pair $(u,v)$ of subobjects of $X$ as on the left hand side is called jointly strongly epic in $\EE$
		$$
		\xymatrix@=6pt{
			&&  W \ar@{ >->}[dd]^w &&  && && && X \ar@{.>}[dd] && \\
			U \ar@{ >->}[rrd]_u \ar@{.>}[rru] &&&&  V \ar@{ >->}[lld]^v  \ar@{ >.>}[llu] &&&& U \ar[rrd]_f \ar@{ >->}[rru]^u &&&&  V \ar[lld]^g \ar[ull]_v\\
			&& X && && && && Z &&
		}
		$$
		when any subobject $w$ of $X$ containing $u$ and $v$ is necessarily an isomorphism.
	\end{defi}
	When $(u,v)$ is jointly strongly epic, there is at most one vertical factorization making the right hand side diagram commute: given any pair of such maps, take their equilizer.
	We are now  going to show that $\PQd\EE$ satisfies the Mal'tsev property, but only relatively to a certain class $\Theta$ of split epimorphims of $\EE$.
	\begin{defi}
		Let $\EE$ be a finitely complete category, and $\Theta$ a class of split epimorphisms of $\EE$ containing the isomorphisms and stable under pullbacks along any map in $\EE$. Then $\EE$ is said to be a $\Theta$-Mal'tsev category when the property of the diagram $(M)$ is only demanded for the split epimorphism $(g,t)$ belonging to $\Theta$.
	\end{defi}
	We are ready to define a pullback stable class $\Theta$ of split epimorphisms in $\PQd$:
	\begin{defi}
		A element $x\in (X,\trio)$ is said to be acupuncturing when the map $y\mapsto x\trio y$ is bijective.
		We call acupuncturing a split epimorphism $(f,s): (X,\trio)\to (Y,\trio)$ when for any $y\in Y$, the element $s(y)$ is acupuncturing in the prequandle $f^{-1}(y)$.
	\end{defi}
	A prequandle $(X,\trio)$ is a latin ILO setting if and only if any element is acupuncturing. We shall denote by $\Theta$ the class of acupuncturing split epimorphisms in $\PQd$ and by $\LPq$ the subcategory of $\PQd$ whose objects are the  latin prequandles. 
	\begin{prop}
		The class $\Theta$ is stable under pullback and contains the isomorphims.
	\end{prop}
	\proof
	Given any pullback in $\PQd$ as on the right hand side with $(f',s')\in \Theta$:
	$$
	\xymatrix@=8pt{
		f^{-1}(y)\ar[dd] \ar@{ >->}[rr]  && X \ar[dd]_f \ar[rr]^k && X'\ar[dd]_{f'}\\
		&&&&\\
		1 \ar@<-1ex>[uu]_{s(y)} \ar[rr]_{y} && Y \ar@<-1ex>[uu]_{s} \ar[rr]_h &&  Y' \ar@<-1ex>[uu]_{s'}
	}
	$$
	we have $f^{-1}(y)=f'^{-1}(h(y))$ and $s'(h(y))=k(s(y))$. Consequently $s(y)$ is acupuncturing in $f^{-1}(y)$ since $s'(h(y))$ is acupuncturing in $f'^{-1}(h(y))$; so, $(f,s)$ is in $\Theta$. The second point is straightforward.
	\endproof
	Suppose $(g,t): (Y,\trio) \splito (Z,\trio) \in \Theta$. The fact that $t(z)$ is acupucturing in $g^{-1}(z)$ defines a bijective mapping $\theta_z: g^{-1}(z)\to g^{-1}(z)$ such that, $\forall y\in  g^{-1}(z)$, we get $t(z)\trio \theta_z(y)=y$.
	\begin{prop}
		The category $\PQd$ is a $\Theta$-Mal'tsev category.
	\end{prop}
	\proof
	Consider any pullback with $(g,t)\in \Theta$:
	$$
	\xymatrix@=8pt{
		X\times_Z  Y \ar[dd]_{p_X} \ar[rr]_{p_Y}  && Y \ar[dd]_g \ar@<-1ex>@{ >->}[ll]_{\iota_Y^s} \\
		&&&& \\
		X \ar@<-1ex>@{ >->}[uu]_{\iota_X^t}  \ar[rr]_{f} && Z\ar@<-1ex>[ll]_{s} \ar@<-1ex>[uu]_{t}
	}
	$$
	and $U\subset X\times_ZY$ a subquandle such that for all $x\in X$, we have $\iota_X^t(x)=(x,tf(x))\in U$ and for all $y\in Y$, we have $\iota_Y^s(y)=(sg(y),y)\in U$. For any $(x,y)\in X\times_ZY$ with $f(x)=z=g(y)$, we get: $(x,y)=( ((s(z)\circ_{\trio}x)\trio s(z),t(z)\trio \theta_z(y)))$\\ $=((s(z)\circ_{\trio}x),t(z))\trio (s(z),\theta_z(y))$
	$=\iota_X^t(s(z)\circ_{\trio}x)\trio\iota_Y^s(\theta_z(y))$\\ where $\circ_\trio$ is the adjoint operation of $\trio$. So, any $(x,y)\in X\times_ZY$ belongs to $U$.
	\endproof 
	Let us briefly mention two meaningful consequences of this $\Theta$-Mal'tsev  property.
	Let $(d_0^R,d_1^R): R\rightarrowtail X\times X$ be  any internal reflexive relation in $\PQd$. Denote by $s_0^R: X\to R$ the map characterizing the reflexivity. Then define a reflexive relation $R$ as {\em acupuncturing} when the split epimorphism $(d_0^R,s_0^R)$ is a acupuncturing.
	\begin{prop}
		Any internal acupuncturing reflexive relation $R$ is transitive.\\
		An internal equivalence relation is acupuncturing if and only if any class $\bar x$ is a latin prequandle.
	\end{prop}
	\proof
	Consider the following pullback in $\PQd$:
	$$
	\xymatrix@=8pt{
		R\times_X R \ar[dd]_{d_0^R} \ar[rr]_{d_2^R}  && R \ar[dd]_{d_0^R} \ar@<-1ex>@{ >->}[ll]_{s_1^R} \\
		&&&& \\
		R \ar@<-1ex>@{ >->}[uu]_{s_0^R}  \ar[rr]_{d_1^R} && X \ar@<-1ex>[ll]_{s_0^R} \ar@<-1ex>[uu]_{s_0^R}
	}
	$$
	$R\times_X R$ is the prequandle of the pairs $(uRv,vRw)$. Let us consider, the subquandle $U\subset R\times_X R$ whose objects satisfies $uRw$. The map $s_0^X: R\to r\times_XR$ is define by $s_0^R(uRv)=(uRv,vRv)$ and map $s_0^X: R\to R\times_XR$ is defined by $s_1^R(uRv)=(uRu,uRv)$. The two maps factor through $U$. Since the reflexive relation $R$ is acupuncturing, $U=R\times_X R$ and the reflexive relation $R$ is transitive.
	
	Saying that the equivalence  relation $R$ is acupuncturing is saying that, for any  $x_0\in X$, the element $(x_0,x_0)$ is acupuncturing in the set $\{(x_0,x)/\; x_0Rx\}$. This is equivalent to saying that the element $x_0$ is  acupuncturing in the class $\overline{x_0}$. It is then true for any $x\in\overline{x_0}$, and $\overline{x_0}$ is a latin prequandle. 
	\endproof 
	\begin{prop}
		The subcategory $\LPq\hookrightarrow \PQd$ is stable under finite limits.
		So, $\LPq$ is a Mal'tsev category.
	\end{prop}
	\proof
	It is easy to check that $\LPq$ is stable under product and equalizer, and thus under pullback. So, any split epimorphism in $\LPq$ is acupuncturing and $\LPq$ is a Mal'tsev category.
	\endproof
	Since $\LPq$ is a variety, the classical Mal'tsev characterization implies the existence of a ternary term $p(x,y,z)$ such that $p(x,x,y)=y=p(y,x,x)$.
	If we denote $d$ the binary law of which $\trio$ is adjoint, we can set $p(x,y,z)=(y\circ x)\trio d(z,y)$. We indeed check $p(x,y,y)=(y\circ x)\trio d(y,y)=(y\circ x)\trio y=x$, while \\
	$p(x,x,z)=(x\circ x)\trio d(z,x)=x\trio d(z,x)=z$.
	\begin{prop}\label{latal}
		Let $\EE$ be a pointed protomodular category. Then $\LPq\EE$ is isomorphic to $\LPq(\Ab\EE)$.
		When $\AA$ is additive, $\LPq\AA$ is isomorphic to the full subcategory $\LAlex\AA$ of $\Alex\AA$ whose objects are the pairs $(X,f)$ such that both $f$ and $Id_X-f$ are isomorphisms.
	\end{prop}
	\proof
	Given $(X,\trio)$ any latin prequandle is $\EE$. Then $p(x,y,z)=(y\circ x)\trio d(z,y)$ becomes an internal Mal'tsev operation in the pointed protomodular category $\EE$. So, the unitary magma structure on $X$ given by $(x,y)\mapsto p(x,1,y)$ makes the object $X$ an abelian object in $\EE$ and $(X,\trio)$ a latin prequandle inside $\Ab\EE$.
	
	When $\AA$ is additive, and $f$ is an automorphism on $X$, the prequandle $(X,\trio_f)$ is a latin one if and only if $x\mapsto 0\trio_fx=f(0-x)+x=(Id_X-f)(x)$ is an isomorphism.
	\endproof
	
	\subsection{Autonomous latin prequandles}
	\begin{defi}
		A internal prequandle $(X,\trio)$ in $\EE$ is called autonomous when $(x\trio x')\trio (y\trio y')=(x\trio y)\trio (x'\trio y')$.
	\end{defi}
	Any Alexander quandle $\Al(X,+)$ is autonomous. With $y=y'$, we check that any autonomous prequandle is actually a quandle. So, autonomous prequandles and autonomous quandles coincide. Denote by $\APq\EE\subset \PQd\EE$ the subcategory of internal autonomous prequandles.
	
	Now, given any pair $(\circ,*)$ of binary operations on a set $X$, let us define the relation  $\circ\HH *$: "$\circ$ is a $*$-homomorphism". It is symmetric and such that: $\circ \HH *$ if and only if $\circ \HH *^{op}$. Since in $\PQd$ a map is a $\trio$-homomorphism if and only if it is a $\circ$-homomorphism, in the category $\APq$, from $\trio\HH\trio$, we get $\trio \HH \circ$, $\circ \HH \trio$ and then $\circ \HH \circ$.
	
	Consider now the subcategory $\ALPq\EE\subset \APq\EE$ of autonomous latin prequandles. From $\trio \HH \trio$, we get $\trio \HH d$. Accordingly, from $\trio\HH\trio$, $\circ \HH\trio$ and $d\HH\trio$, the ternary term $p(x,y,z)=(y\circ x)\trio d(z,y)$ becomes internal to $\ALPq\EE$. In this way, any object $(X,\trio)\in \ALPq\EE$ is endowed with a natural ternary Mal'tsev  operation . Whence, immediately, the following
	\begin{prop}\label{auto}
		The category $\ALPq\EE$ is a naturally Mal'tsev category.
	\end{prop}
	This notion was introduced in \cite{Jon} as a category $\AA$ in which any object $X$ is endowed with a natural Mal'tsev operation and, strenghtening the  notion of Mal'tsev category, characterized by the fact that any internal graph is endowed with a unique  groupoid structure. A pointed category is a naturally Mal'tsev one if and only if it is additive.
	
	\section{Semi-direct index and hyperindex}\label{sd}
	
	Let $U: \CC \to \DD$ be now a left exact functor between finitely complete {\em pointed} categories. Let $\Pt_U: \Pt\CC \to \Pt\DD$ be its natural extension. Denote by $W_U: \Pt\CC \to \Pt\DD$ the functor associating with any split epimorphism $(f,s): X \splito Y$ in $\CC$ the split epimorphism $(p_{U(Y)},\iota_{U(Y)}): U(Y)\times U(\Ker f) \splito U(Y)$ in $\DD$. 
	\begin{defi} A semi-direct index for $U$ is a natural isomorphism $\rho$ between $\Pt_U$ and $W_U$ given by an isomorphism $\rho_{(f,s)}$ making the following left hand side diagram commute in $\DD$ for any split epimorphism $(f,s):X\splito Y$ in $\CC$:
		$$
		\xymatrix@=8pt{
			U(X) \ar[dd]_{U(f)}  \ar[rr]^<<<<<<{\rho_{(f,s)}}  && {\; U(Y)\times  U(\Ker f)} \ar[dd]_{p_{U(Y)}} && U(X) \ar[dd]_{U(f)}  \ar[rr]^<<<<<<{\gamma_{(f,s)}} && U(\Ker f) \ar[dd]\\
			\\
			U(Y) \ar@<-1ex>[uu]_{U(s)}  \ar[rr]_{=}  && U(Y) \ar@<-1ex>[uu]_{\iota_{U(Y)}} && U(Y) \ar@<-1ex>[uu]_{U(s)} \ar[rr] && 1 \ar@<-1ex>[uu]
		}
		$$
		or, equivalently, it is given a natural map $\gamma_{(f,s)}: U(X) \to U(\Ker f)$ in $\DD$ making the right hand side square a pullback of split epimorphisms; namely, a pullback satisfying $\gamma_{(f,s)}.U(s)=0 \; (*)$ in $\DD$.
	\end{defi}
	\begin{defi} A semi-direct hyperindex for $U$ is a semi-direct index $\rho$ making the following left hand side diagram commute in $\DD$ for any split epimorphism $(f,s):X\splito Y$ in $\CC$:
		$$
		\xymatrix@=8pt{
			U(Y) \ar@<1ex>[rr]^{U(s)} \ar[dd]_{=} && U(X) \ar[rr]_{\gamma_{(f,s)}} \ar[ll]^{U(f)} \ar[dd]_{\rho_{(f,s)}}  && U(\Ker f) \ar[dd]^{=}\ar@<-1ex>[ll]_{U(k_f)} && U(X) \ar[dd]_{U(f)}  \ar@<-1ex>[rr]_<<<<<<{\gamma_{(f,s)}} && U(\Ker f)\ar[ll]_{U(k_f)} \ar[dd] \\
			&& \\
			U(Y) \ar@<1ex>[rr]^<<<<<{\iota_{U(Y)}} && U(Y)\times U(\Ker f)\ar[rr]_>>>>>{p_{U(K)}} \ar[ll]^>>>>>{p_{U(Y)}}   &&  U(\Ker f) \ar@<-1ex>[ll]_<<<<<{\iota_{U(K)}} && U(Y) \ar@<-1ex>[uu]_{U(s)} \ar@<-1ex>[rr] && 1 \ar@<-1ex>[uu]  \ar[ll]
		}
		$$
		or, equivalently, making the following right hand side square a pullback of split epimorphisms; namely, adding to the index $\gamma$ the axiom  $\gamma_{(f,s)}.U(k_f)=Id_{U(\Ker f)}$.
	\end{defi}
	
	\medskip
	
	\noindent\textbf{Examples}\\
	1) Considering the forgetful functor $V:\Gp \to \Set_*$ from the category of groups to the category of pointed sets, given any split epimorphism $(f,s): X\splito Y$ in $\Gp$, the map $\gamma_{(f,s)}: V(X)\to V(\Ker f)$ defined by $\gamma_{(f,s)}(x)=sf(x)^{-1}.x$ is a semi-direct hyperindex for $U$ whose inverse is given by $(y,k)\mapsto s(y)\circ k$.\\
	2) Let us start with any finitely complete category $\EE$. Denoting the fiber $\Pt_1\EE$ by $\EE_*$ and the category of internal groups in $\EE$ by $\Gp\EE$, it is clear that the previous formula determines a semi-direct hyperindex for the forgetful functor $V_\EE:\Gp\EE \to \EE_*$.\\
	3) Let $R$ be a ring, and $\CC$ be the category of any kind of $R$-algebras, the forgetful functor $U:\CC\to \Ab$ has a semi-direct hyperindex with the group homomorphism $\gamma_{(f,s)}: U(X)\to U(\Ker f)$ defined by $\gamma_{(f,s)}(x)=x-sf(x)$.\\
	4) Consider the category $\SSlom\DD$ of internal hyper-S\l om\'inski settings  in a finitely complete category $\DD$. Then the forgetful functor $\Sigma_\DD: \SSlom\DD\to D_*$ has, for any split epimorphism $(f,s)$ in $\SSlom\DD$, a semi-direct index with the map $\gamma^\Sigma$ defined by $\gamma^\Sigma_{(f,s)}(x)=d(x,sf(x))$, the inverse of $(f,\gamma^\Sigma_{(f,s)}): X\to Y\times  \Ker f$ being produced by $(y,k)\mapsto s(y)\circ k$.\\
	5) Consider the category $\HSt\DD$ of internal hypersubtractions in a finitely complete category  $\DD$. Then the forgetful functor $H_\DD: \HSt\DD\to D_*$ has, for any split epimorphism $(f,s)$ in $\HSt\DD$, a semi-direct hyperindex with the map $\gamma^{HS}$ defined by $\gamma^{HS}_{(f,s)}(x)=d(x,sf(x))$, the inverse of $(f,\gamma^{HS}_{(f,s)}): X\to Y\times  \Ker f$ being produced in the same way by $(y,k)\mapsto s(y)\circ k$.\\
	6) Let $\Di\Gp$ be the category of digroups, namely of quadruples $(X,*,\circ, 1)$ of a set $X$ endowed with two group structures having same unit element, and $\SkB$ the subcategory of {\em left skew braces} \cite{GV} where the two laws are related by the axiom: $a\circ (b * c) = (a\circ b)*a^{-*}* ( a\circ c)$, where $a^{-*}$ denotes the inverse of $a$ in the group $(X,*)$. They are both protomodular, see \cite{BFP}. Clearly the two possible forgetful functors towards $\Gp$ composed with the forgetful functor $V: \Gp\to \Set_*$ produce two absolutely independent (extrinsic) semi-direct hyperindexes. So, the hyperindexes of a functor are far from being unique. It is shown in \cite{BSk} that the thorough description of the split epimorphisms in $\Di\Gp$ and $\SkB$ actually requires one more item than these only two hyperindexes.
	
	\begin{lemma}
		Let $U: \CC \to \DD$ be a left exact functor between finitely complete pointed categories which is endowed a semi-direct index $\rho$. If, in addition, the functor $U$ is conservative, then the category $\CC$ is protomodular, without any assumption on the category $\DD$.
	\end{lemma}
	\proof
	Consider any morphism $h$ of split epimorphisms as on the right hand side:
	$$
	\xymatrix@=20pt{
		\Ker f \ar[d] \ar[r]^{k_f} & X \ar[d]_f \ar[r]^h  & X' \ar[d]_{f'} \\
		1 \ar@<-1ex>@{ >->}[u]  \ar@{ >->}[r]_{\alpha_Y} & Y \ar@{=}[r] \ar@<-1ex>@{ >->}[u]_s & Y \ar@<-1ex>@{ >->}[u]_{s'} 
	}
	$$
	Saying that $\alpha_Y^*(h)$ is an isomorphism is equivalent to saying that the whole rectangle  is a pullback. Now since $\rho$ is underlying a natural transformation, we have the following commutative diagram in $\DD$:
	$$
	\xymatrix@=20pt{
		U(X) \ar[d] \ar[r]_{U(h)} \ar@<2ex>[rr]^{\gamma_{(f,s)}} & U(X') \ar[d] \ar[r]_{\gamma_{(f',s')}}  & \Ker f \ar[d] \\
		U(Y) \ar@<-1ex>@{ >->}[u]  \ar@{=}[r] & U(Y) \ar[r] \ar@<-1ex>@{ >->}[u] & 1 \ar@<-1ex>@{ >->}[u] 
	}
	$$
	where the right hand side square and the whole rectangle are pullbacks. Accordingly the left hand side square is a pullback or, equivalently, $U(h)$ is an isomorphism. Since  $U$ is conservative, so is $h$; and the base-change $\alpha_Y^*$ is conservative as well.
	\endproof
	
	The fourth and the fith examples here above will furnish us characterizations of the functors with semi-direct index and hyperindex.
	\begin{theo}
		Let $U: \CC \to \DD$ be a left exact functor with semi-direct index $\rho$ between finitely complete pointed categories. Then the functor $U$ can be factorized through the forgetful functor $\Sigma_\DD: \SSlom\DD \to \DD$  by a functor $\bar U: \CC \to \SSlom\DD$ in such a way that $\rho=\rho^{\Sigma}.\bar U$. When it has a semi-direct hyperindex, the functor $U$ can be factorized through the forgetful functor $H_\DD: \HSt\DD \to \DD$ by a functor $\bar U: \CC \to \HSt\DD$ in such a way that $\rho=\rho^{HS}.\bar U$.
	\end{theo} 
	\proof
	Given any object $X$ in $\CC$, consider the split epimorphism $(p_1^X,s_0^X):X\times X \splito X$, where $p_1^X$ is the second projection and $s_0^X:X\into X\times X$ is the diagonal. Then consider the following isomorphism given by the semi-direct index:
	$$
	\xymatrix@=8pt{
		U(X)\times U(X) \ar[dd]_{p_1^{U(X)}}  \ar[rr]^<<<<<<{\rho_{(p_1^X,s_0^X)}}  && {\; U(X)\times U(X)} \ar[dd]_{p_0^{U(X)}} \\
		&& \\
		U(X) \ar@<-1ex>[uu]_{s_0^{U(X)}}  \ar[rr]_{=}  && U(X) \ar@<-1ex>[uu]_{\iota_0^{U(X)}} 
	}
	$$
	According to Definition \ref{sslom}, the  map $\gamma_{(p_1^X,s_0^X)}:U(X)\times U(X) \to U(X)$, which will be denoted by $d_X$ for sake of simplicity, gives $U(X)$ a structure of hyper-S\l om\'inski setting. Whence the factorization $\bar U: \CC \to \SSlom\DD $, defined by $\bar U(X)=(U(X), d_X,\alpha_X)$ through the forgetful functor $\SSlom\DD \to \DD$; it clearly  satisfies $\rho_{(U(f),U(s))}^\Sigma=\rho_{(f,s)}$. When this index is a hyperindex, the map $d_X$ gives $U(X)$ a structure of hypersubtraction and produces a factorization $\bar U: \CC \to \HSt\DD $, defined by $\bar U(X)=(U(X), d_X,\alpha_X)$ through the forgetful functor $\HSt\DD \to \DD$; it clearly satisfies $\rho_{(U(f),U(s))}^{HS}=\rho_{(f,s)}$.
	\endproof
	The examples 3) of hyperindex achieve their full meaning knowing that $\Ab=\UMag(\HSt)=\HSt(\UMag)$.

\end{document}